\documentclass{birkjour}
\usepackage{amsmath}
\usepackage{amssymb}
\usepackage{amsthm}
\usepackage{color}
\usepackage{hyperref}
\usepackage{eqlist}
\usepackage{enumerate}
\usepackage[mathscr]{eucal}
\hypersetup{
     colorlinks   = true,
     citecolor    = blue
}
\hypersetup{linkcolor=blue}

\theoremstyle{definition}
\newtheorem{theorem}{Theorem}[section]
\newtheorem{definition}[theorem]{Definition}
\newtheorem{proposition}[theorem]{Proposition}
\newtheorem{lemma}[theorem]{Lemma}
\newtheorem{corollary}[theorem]{Corollary}
\newtheorem{remark}[theorem]{Remark}

\numberwithin{equation}{section}
\begin{document}
\title[Pointwise Semi-Slant Submanifolds]
{Pointwise Semi-Slant Warped Product Submanifold in a Lorentzian Paracosymplectic Manifold}
\author[S. K. Srivastava and A. Sharma]
{S. K. S\lowercase{rivastava and} A. S\lowercase{harma}}
\address{Department of Mathematics,\\
                   Central University of Himachal Pradesh,\\
                   Dharamshala-176215,\\
                   Himachal Pradesh, INDIA.}
\email{\textcolor[rgb]{0.00,0.00,0.84}{sachink.ddumath@gmail.com}}
\email{\textcolor[rgb]{0.00,0.00,0.84}{anilsharma3091991@gmail.com}}
\thanks {S. K. Srivastava: partially supported through the UGC-BSR Start-Up-Grant vide
their letter no. F.30-29/2014(BSR). A. Sharma: supported by Central University of Himachal Pradesh through the Research fellowship for Ph.D} 
\begin{abstract}
Recently Y{\"u}ksel et. al. \cite{SY} shows that there doesn't exist any proper semi-slant warped product submanifolds in a Lorentzian paracosymplectic manifold. In the present article, we first define and give preparatory lemmas for a new generalize class of semi-slant submanifolds called pointwise semi-slant submanifolds in a Lorentzian paracosymplectic manifold, and then we ensure by presenting some existence results and a non-trivial characterization theorem that there exist a pointwise semi-slant warped product submanifolds in a Lorentzian paracosymplectic manifold counter to warped product semi-slant submanifolds in a Lorentzian paracosymplectic manifold. 
\end{abstract}
\maketitle
\section{Introduction}
The premise of Lorentzian almost paracontact manifold (introduced by K. Matsumoto \cite{KM}) and warped product submanifolds one of the most effective generalization of pseudo-Riemannian products (initiated by Bishop-O'Neill, B \cite{RB}), has recognized various significant contributions in Lorentzian geometry (or pseudo-Riemannian geometry), and has been successfully employed in different models of space-time, general relativity and black holes (c.f., \cite{Beem1981, BYCBook, Kreitler, OB}). Because of its numerous application to mathematical physics, several researcher found interest and studied the geometry of Lorentzian almost paracontact manifold and warped product submanifold in different settings (see; \cite{PA, BYCsurvey, Duggal, MR, PO, MM}). 

On the other hand, the concept of pointwise slant submanifold was introduced by Chen-Garay \cite{CG} as the natural generalization of slant submanifolds \cite{BYC}. Such submanifolds were earlier studied by Etayo \cite{Etayo} with the name quasi-slant submanifold in almost Hermitian manifolds. Later on, Sahin \cite{Sahin} continued the study of pointwise slant submanifold by presenting a new class of submanifolds called warped product pointwise semi-slant submanifolds in K\"{a}hlerian manifolds. Recently, Park \cite{Park, Park2} and Balgeshir \cite{MBK} extended the notion of pointwise slant, pointwise semi-slant submanifolds and pointwise almost $h$-semi-slant submanifolds along with its warped products aspects in almost contact and quaternionic Hermitian settings.  Motivated by the works of these, in this research we introduced the pointwise semi-slant submanifolds in Lorentzian almost paracontact manifolds which can be considered as the generalization of slant, pointwise slant, semi-invariant, semi-slant submanifolds and investigate the warped aspects for such submanifold. 

The organization of article is as follows. In Sect. \ref{pre}, we recall some basic informations about  Lorentzian paracosymplectic manifold. Subsect. \ref{sub}, \ref{Wpsub} and \ref{psub}, includes some basic formulas, definitions of warped product submanifold, pointwise slant submanifold and some characterization results for such sumanifolds. Sect. \ref{Pss}, deals with the construction of pointwise semi-slant submaniold along with the necessary and sufficient conditions for the distributions allied to the characterization of a pointwise semi-slant submanifold to be involutive and totally geodesic foliation. In Sect. \ref{Psswp}, we first define pointwise semi-slant warped product submanifold $M$, and then give existence and nonexistence results for such warped product submanifolds. We also, obtain a characterization theorem for warped product submanifold of the form $M_{T}\times_{f}M_{\theta}$ with $\xi \in \Gamma(M_{T})$ where, $M_{T}$ and $M_{\theta}$ are invariant and  pointwise proper slant  submanifolds on $M$, respectively and $f$ is a non-constant positive smooth function in a Lorentzian paracosymplectic manifold.

\section{Preliminaries}\label{pre}
Let $\bar{M}^{2m+1}$ be a $2m+1$-dimensional $C^{\infty}$ manifold. Then $\bar{M}^{2m+1}$ is said to have an almost paracontact structure $(\phi ,\xi ,\eta)$, if there exist on $\bar{M}^{2m+1}$ a tensor field $\phi$ of type $(1, 1)$, a smooth vector field $\xi$, and a $1$-form $\eta$  satisfying 
\begin{align}
&\phi ^{2} =I+\eta \otimes \xi ,\quad \eta (\xi )=-1\label{phieta}\\
&\phi\xi = 0, \quad \eta\circ\phi = 0 \quad {\rm and \quad rank}(\phi)=2m.\label{phixi}
\end{align}
where $I$ is the identity transformation. 
If the manifold $\bar{M}^{2m+1}$ has an almost paracontact structure $(\phi ,\xi ,\eta)$ and admits a Lorentzian metric $g$ of type $(0, 2)$ on $\bar{M}^{2m+1}$ such that 
\begin{align}
g(\phi X,\phi Y)=g\left(X,Y\right)+\eta (X)\eta (Y),\label{metric}
\end{align}
where signature of $g$ is necessarily $(1,\,2m)\, or\, (2m,\,1)$ for any vector fields $X$ and $Y$; then the quadruple $(\phi, \xi, \eta, g)$ is called an Lorentzian almost paracontact structure and the manifold $\bar{M}^{2m+1}$ equipped with Lorentzian almost paracontact structure is called an Lorentzian almost paracontact manifold $\bar{M}^{2m+1}(\phi, \xi, \eta, g)$. The Lorentzian metric $g$ makes $\xi$ a timelike unit vector field, that is, $g(\xi ,\xi ) = -1$ (see, \cite{KM, MR}). With respect to $g$, $\eta$ is metrically dual to $\xi$, that is $g(X,\xi)=\eta(X)$. 
In light of Eqs. \eqref{phieta}, \eqref{phixi} and \eqref{metric}, we deduce that
\begin{align}\label{symphi}
g(\phi X,Y)-g(X,\phi Y)=0, 
\end{align}
for any $X,Y \in \Gamma(T\bar{M})$. Here $\Gamma(T\bar{M}^{2m+1})$ is the tangent bundle of $\bar{M}^{2m+1}$.  Finally, the fundamental $2$-form $\Phi$ on $\bar{M}^{2m+1}$ is given by
 \begin{align}\label{PHI} 
g(X,\phi Y)=\Phi(X,Y). 
 \end{align}
 Moreover,
  \begin{align}\label{nabPHI} 
(\bar{\nabla }_{Z}\Phi)(X,Y)=g((\bar{\nabla }_{Z}\phi)X, Y)=(\bar{\nabla }_{Z}\Phi)(Y, X),
 \end{align}
 for any  $X,Y, Z \in \Gamma(T\bar{M})$, $\bar{\nabla }$ is the Levi-Civita connection on $\bar{M}^{2m+1}(\phi, \xi, \eta, g)$.
\begin{definition} 
 A Lorentzian almost paracontact manifold $\bar{M}^{2m+1}(\phi, \xi, \eta, g)$ is called \cite{PO, SY} {\it Lorentzian paracosymplectic} $\bar{M}^{2m+1}$, if the forms $\eta$ and $\Phi$ are parallel with respect to the Levi-Civita connection $\bar{\nabla }$ on  $\bar{M}^{2m+1}(\phi, \xi, \eta, g)$, i.e., 
 \begin{align}\label{pcmdef}
\bar{\nabla } \eta=0\quad{\rm and}\quad \bar{\nabla }\Phi=0
 \end{align}
 for any  $X,Y \in \Gamma(T\bar{M})$.
\end{definition}
From the direct consequence of above definition, Eq. \eqref{phixi} and covariant differentiation formula, we have the following result;
\begin{lemma}\label{lemnabxi}
On a Lorentzian paracosymplectic manifold $\bar{M}^{2m+1}$ such that the structure vector field $\xi \in \Gamma(T\bar{M})$, we have
 \begin{align}
 \bar\nabla_{X}\xi =0,
 \end{align}
 for any $X \in \Gamma(T\bar{M}).$
\end{lemma}
\subsection{Geometry of submanifolds}\label{sub}
Let $M$ be a real submanifold immersed in a Lorentzian paracosymplectic manifold $\bar{M}^{2m+1}$, we denote by the same symbol $g$ the induced metric on $M$. In this article, we assume that $g$ is non-degenerate (in the sense of \cite{Duggal, OB}). Thus, each tangent space $T_{p}(M)$, for every $p \in M$, is a non-degenerate subspace of $T_{p}(\bar{M)}$ such that $T_{p}(\bar{M)} = T_{p}(M)\oplus T_{p}(M)^{\bot}$, where $T_{p}(M)^{\bot}$  denotes the normal space of $M$. If $\Gamma (TM^{\bot })$ indicate the set of vector fields normal to $M$ and $\Gamma(TM)$ the sections of tangent bundle $TM$ of $M$, then the Gauss-Weingarten formulas are given by, respectively,
\begin{align} 
\bar{\nabla }_{X} Y&=\nabla _{X} Y+h(X,Y), \label{gauss}\\
\bar{\nabla }_{X} \zeta &=-A_{\zeta} X+\nabla _{X}^{\bot }\zeta,\label{weingarten}
\end{align}
for any $X,Y \in \Gamma(TM)$ and $\zeta \in \Gamma(TM^{\bot })$, where $\nabla$ is the induced connection, $\nabla ^{\bot }$ is the normal connection on the normal bundle  $\Gamma(TM^{\bot })$, $h$ is the second fundamental form, and the shape operator $A_{\zeta}$ associated with the normal section $\zeta$ is given in \cite{Chensub} by 
\begin{align}
 \label{shp2form} g\left(A_{\zeta} X,Y\right)=g\left(h(X,Y),\zeta\right).
\end{align}
 If we write, for all $X \in \Gamma(TM)$ and $\zeta \in \Gamma(TM^{\bot })$ that
\begin{align} 
\phi X&=tX+nX,\label{phix}\\
\phi \zeta&=t'\zeta+n'\zeta, \label{phin}
\end{align}
where $tX$ (resp., $nX$) is tangential (resp., normal) part of $\phi X$ and $t'\zeta$ (resp., $n'\zeta$) is tangential (resp., normal) part of $\phi \zeta$.  Then the submanifold $M$ is said to be {\it invariant} if $n$ is identically zero and {\it anti-invariant} if $t$ is identically zero.
From Eqs. $\eqref{symphi}$ and $\eqref{phix}$, we obtain for all $X \in \Gamma(TM)$ that
\begin{align} \label{symxty}
g(X,tY)=g(tX,Y). 
\end{align}
A distribution $D$ on a submanifold $M$ is said to be \cite{BYCBook, Duggal} 
\begin{itemize}
\item[$\bullet$] {\it totally geodesic} if its second fundamental form vanishes identically. 
\item[$\bullet$] {\it umbilical} in the direction of a normal vector field $\zeta$ on $M$, if $A_{\zeta} = \lambda Id$, for certain function $\lambda$ on $M$; here $\zeta$ is called a umbilical section. 
\item[$\bullet$] {\it totally umbilical} if $M$ is umbilical with respect to every (local) normal vector field. 
\item[$\bullet$] {\it involutive} if, for all $X, Y \in D, [X, Y] \in D.$
\end{itemize}
Now we have an important results by virtue of Lemma \ref{lemnabxi} and Eq. \eqref{shp2form}, 
\begin{lemma}\label{lemsub}
If $M$ is a isometrically immersed submanifold in a Lorentzian paracosymplectic manifold $\bar{M}^{2m+1}$ such that the structure vector field $\xi \in \Gamma(TM)$, then
 \begin{align}
  &\nabla_{X}\xi =\nabla_{\xi}X=\nabla_{\xi}\xi=0 \,\,{\rm and}\,\, h(X,\xi)=0, \nonumber \\
  &A_{\zeta}\xi=0 \,\,{\rm and}\,\,  A_{\zeta}X\,\bot\,\xi \nonumber
 \end{align}
 for any $X \in \Gamma(TM)$ and $\zeta \in \Gamma(TM^{\bot })$.
\end{lemma}
\subsection{Warped product submanifolds}\label{Wpsub}
\noindent Let $\left(B, g_{B} \right)$ and $\left(F ,g_{F} \right)$ be two pseudo-Riemannian manifolds and ${f}$ be a positive smooth function on $B$. Consider the product manifold $B\times F$ with canonical projections 
\begin{align}\label{cp}
\pi:B \times F\to B\quad{\rm and}\quad \sigma:B \times F\to F.
\end{align}
Then the manifold $M=B \times_{f} F $ is said to be \textit{warped product} if it is equipped with the following warped metric
\begin{align}\label{wmetric}
g(X,Y)=g_{B}\left(\pi_{\ast}(X),\pi_{\ast}(Y)\right) +(f\circ\pi)^{2}g_{F}\left(\sigma_{\ast}(X),\sigma_{\ast}(Y)\right)
\end{align}
for all $X,Y\in \Gamma(TM)$ and `$\ast$' stands for derivation map, or equivalently,
\begin{align}
g=g_{B} +f^{2} g_{F}.
\end{align}
The function $f$ is called {\it the warping function} and a warped product manifold $M$ is said to be {\it trivial} if $f$ is constant. In view of simplicity, we will determine a vector field $X$ on $B$ with its lift $\bar X$ and a vector field $Z$ on $F$ with its lift $\bar Z$ on $M=B \times_{f} F $ \cite{RB}.
\begin{proposition}\label{propmain}\cite{RB}
For $X, Y \in \Gamma(TB)$ and $Z, W \in \Gamma(TF)$, we obtain on warped product manifold $M=B\times_{f} F$ that
\begin{itemize}
\item[(i)]   $\nabla _{X}Y \in \Gamma(TB),$
\item[(ii)]	$\nabla _{X}Z =\nabla _{Z}X=\left(\frac{X{f}}{f} \right)Z,$
\item[(iii)]	$\nabla _{Z}W =\frac{-g(Z, W)}{f} \nabla f,$
\end{itemize}
where $\nabla$ denotes the Levi-Civita connection on $M$ and $\nabla f$ is the gradient of $f$ defined by $g(\nabla f, X)=Xf$.
\end{proposition}
\begin{remark}
It is also important to note that for a warped product $M=B \times_{f}F$; $B$ is totally geodesic and $F$ is totally umbilical in $M$ \cite{RB}. 
\end{remark}
\noindent Now, we prove an important results for later use;
\begin{theorem}\label{thmwp}
Let $\bar{M}^{2m+1}$ be a Lorentzian paracosymplectic manifold. Then there doesn't exist any non-trivial warped product submanifolds $M=B\times_{f}F$ of a paracosymplectic manifold such that $\xi \in \Gamma(TF)$.
\end{theorem}
\begin{proof}
In light of Lemma \ref{lemsub} and Proposition \ref{propmain}, we obtain for any non-degenerate vector fields $X \in \Gamma(TB)$ and $Z \in \Gamma(TF)$ that $X(\ln f)Z=0$. This implies that $f$ is constant function, since $X , Z$ are non-degenerate vector fields in $M$. This completes the proof of the theorem.
\end{proof}
\begin{lemma}\label{wplem}
If $M=B \times_{f} F$ is a non-trivial warped product submanifold of a Lorentzian paracosymplectic manifold $\bar{M}^{2m+1}$ with $\xi \in \Gamma(TB)$, then
\begin{align}
 \xi(\ln f)X=0,
\end{align}
for any non-null vector field $X \in \Gamma(TF)$.
\end{lemma}
\begin{proof}
The proof of the lemma can be directly achieved by virtue of Lemma \ref{lemsub} and Proposition \ref{propmain}.
\end{proof}
\subsection{Pointwise slant submanifolds}\label{psub}
\noindent Following the notion of pointwise slant immersion in \cite{MBK, CG}. We define 
\begin{definition}
A submanifold $M$ of a Lorentzian almost paracontact manifold $\bar{M}^{2m+1}(\phi,\xi,\eta, g)$ is said to be \textit{pointwise slant} 
if at each given point $p \in M$, the \textit{slant angle} or \textit{Wirtinger angle} $\theta(X)$ between $\phi (X)$ and the space $T_{p}M$ is independent of the choice of the non-zero vector $X \in \Gamma(TM)$ linearly independent of $\xi$. In this case, the angle $\theta$ can be viewed as a function on $M$, which is called the {\it slant function} of the pointwise slant submanifold.
\end{definition}
\begin{remark}
A point $p$ in a pointwise slant submanifold is called a {\it totally real point} if its slant function $\theta$ satisfies $\cos \theta = 0$ at $p$. Similarly, a point $p$ is called a {\it complex point} if its slant function satisfies $\sin \theta = 0$ at $p$. A pointwise slant submanifold $M$ of Lorentzian almost paracontact manifold $\bar{M}$ is said to be
\begin{itemize}
 \item[$\bullet$]\textit{totally real} if every point of $M$ is a totally real point.
 \item[$\bullet$] \textit{pointwise proper slant} if it contains no totally real points.
\end{itemize}
\end{remark}
\noindent If we denote the orthogonal distribution to $\xi \in \Gamma(TM)$ by $\mathfrak{D}$ then the orthogonal direct decomposition is given as follows: 
\begin{align}
TM=\mathfrak{D}\oplus\{\xi\}, \nonumber
\end{align}
where, span of the characteristic vector field $\xi$ generates the $1$-dimensional distribution $\{\xi\}$  on $M$.

\noindent Furthermore, we give the following useful characterization of pointwise slant submanifolds in Lorentzian almost paracontact manifolds:
\begin{proposition}\label{chara}
Let $M$ be a submanifold in a Lorentzian almost paracontact manifold $\bar{M}^{2m+1}(\phi,\xi,\eta, g)$ such that $\xi \in \Gamma(TM)$. Then $M$ is pointwise slant if and only if $t^{2}=\cos^2\theta (I+\eta \otimes \xi)$ for some real-valued function $\theta$ defined on the tangent bundle $TM$ of $M$.
\end{proposition} 
\begin{proof}
The proof of the proposition is similar to the proof of Lemma $2.1$ of \cite{CG} for Hermition ambient.
\end{proof} 
\noindent The following corollaries are straight forward consequences of the above result:
\begin{corollary}\label{corsub} 
Let $\mathfrak{D}_{\theta}$ be a distribution on $M$. Then $\mathfrak{D}_{\theta}$ is pointwise slant if and only if there exists a function $\theta$ such that $(tP_{\theta})^2 Z = \cos^2\theta\, Z $ for $Z \in \Gamma(\mathfrak{D}_{\theta})$, where $P_{\theta}$ denotes the orthogonal projection on $\mathfrak{D}_{\theta}$. 
\end{corollary}
\begin{corollary}
If $M$ is a pointwise slant submanifold and $\mathfrak{D}_{\theta}$ a pointwise slant distribution on $M$ such that $\xi \in  \Gamma(TM)$, then
\begin{align} 
g(tZ,tW)&=cos^{2}\theta\{\eta(Z)\eta(W)+g(Z,W)\}, \label{gtcos}\\ 
g(nZ,nW)&=sin^{2}\theta\{\eta(Z)\eta(W)+g(Z,W)\}, \label{gnsin}
\end{align}
for any $Z,W\in \Gamma(TM)$. 
\end{corollary}
\section{Pointwise semi-slant submanifolds}\label{Pss}
Analogous to \cite{Sahin} in this section, we define and study pointwise semi-slant submanifolds in a Lorentzian almost paracontact manifold $\bar{M}^{2m+1}$.  We also, derive important results and deduce the geometry of leaves of the involutive distributions involved with the definition of such submanifolds.
\begin{definition}
Let $M$ a real submanifold of a Lorentzian almost paracontact manifold $\bar{M}^{2m+1}(\phi,\xi,\eta, g)$. Then we say that $M$ is a {\it pointwise semi-slant submanifold}, if it is furnished with the pair of complimentary distribution $(\mathfrak{D}_{T},\mathfrak{D}_{\theta})$  satisfying the conditions:
\begin{itemize}
\item[(i)] $TM = \mathfrak{D}_{T}\oplus \mathfrak{D}_{\theta}\oplus \{\xi\}$,
\item[(ii)] the distribution $\mathfrak{D}_{T}$ is invariant under $\phi$, i.e., $\phi(\mathfrak{D}_{T})\subseteq \mathfrak{D}_{T}$ and
\item[(iii)] the distribution $\mathfrak{D}_{\theta}$ is pointwise slant distribution with slant function $ \theta$. 
\end{itemize}
A pointwise semi-slant submanifold is {\it proper} if $\mathfrak{D}_{T} \neq \{0\}$ and $\theta$ is not a constant. Furthermore, we say a pointwise semi-slant submanifold {\it mixed geodesic} if the second fundamental form $h$ of $M$ satisfies $h(\mathfrak{D}_{T}, \mathfrak{D}_{\theta}) = 0$. 
\end{definition}
\noindent In particular, we have the following:
\begin{enumerate}
    \item [(i).] If $\mathfrak{D}_{T}=\{0\}$ and $\theta=\pi/2$, then $M$ is an anti-invariant submanifold \cite{PA, SY}.
    \item [(ii).] If $\mathfrak{D}_{\theta}=\{0\}$, then $M$ is an invariant submanifold \cite{PA, SY}.
    \item [(iii).] If $\mathfrak{D}_{T}=\{0\}$ and $\mathfrak{D}_{\theta} \neq \{0\}$ with $\theta$ globally constant such that $\theta \in (0, \pi/2)$,  then $M$ is a proper slant submanifold \cite{PA}.
   \item [(iv.)] If $\mathfrak{D}_{T} \neq \{0\}$ and $\mathfrak{D}_{\theta}\neq \{0\}$ such that slant angle $\theta=\pi/2$, then $M$ is a semi-invariant submanifold \cite{MM}.
   \item [(v).] If $\mathfrak{D}_{T} \neq \{0\}$ and $\mathfrak{D}_{\theta}\neq \{0\}$ such that slant angle $\theta$ satisfies that $\theta \in (0, \pi/2)$ is independent of point and vector fields on $M$, then $M$ is a proper semi-slant submanifold \cite{LC, SY}.
    \item [(vi).] If $\mathfrak{D}_{T}=\{0\}$ and $\theta$ is a slant function,  then $M$ is a pointwise slant submanifold \cite{MBK}.
\end{enumerate}
Let us consider that $M$ be a pointwise semi-slant submanifold of a Lorentzian paracosymplectic manifold $\bar{M}^{2m+1}$. If  $\mathcal{P}_{T}$ and $\mathcal{P}_{\theta}$ denoted the projections on the distributions $\mathfrak{D}_{T}$ and $\mathfrak{D}_{\theta}$, respectively. Then we can write for any $ Z \in \Gamma(TM)$ that 
\begin{align}\label{semi-phixproj}
Z= \mathcal{P}_{T} Z + \mathcal{P}_{\theta} Z +\eta(Z)\xi. 
\end{align}
Previous equation by operating $\phi$ and Eqs. \eqref{phixi}, \eqref{phix}, becomes $\phi Z =t\mathcal{P}_{T}Z +t\mathcal{P}_{\theta}Z +n\mathcal{P}_{\theta} Z.$ Thus, from previous expression, we conclude that $t\mathcal{P}_{T}Z \in \Gamma(\mathfrak{D}_{T})$, $n\mathcal{P}_{T} X=0,$ and $t\mathcal{P}_{\theta} X \in  \Gamma(\mathfrak{D}_{\theta}), \quad  n\mathcal{P}_{\theta} X \in \Gamma(TM^{\bot}).$ Using Eq. \eqref{phix} and above expressions in Eq. \eqref{semi-phixproj}, we deduce that $tZ = t\mathcal{P}_{T}Z+t\mathcal{P}_{\theta} Z, \quad nZ =n\mathcal{P}_{\theta} Z,$ for any $Z \in \Gamma(TM)$. Since, $\mathfrak{D}_{\theta}$ is pointwise slant distribution, by the consequences of  Corollary \ref{corsub}, we obtain that
\begin{equation}
t^2 Z= ({\cos}^2\,\theta)Z,\label{semi-tsqr}
\end{equation}
for any $Z \in \Gamma(\mathfrak{D}_{\theta})$ and some real-valued function $\theta$ defined on $M$.

\noindent Now, by virtue of above construction, we have the following characterization result for pointwise semi-slant submanifold:
\begin{lemma}\label{semi-lem1} 
If $M$ is a proper pointwise semi-slant submanifold of a Lorentzian paracosymplectic manifold $\bar{M}^{2m+1}$ such that $\xi \in \Gamma(TM)$, then
\begin{align}
g(tZ, tW)&=\cos^2\theta \, g(\phi Z, \phi W) \label{semi-gtcos} \\
g(nZ,  nW)&=\sin^2 \theta \, g(\phi  Z, \phi W) \label{semi-gnsin}
\end{align}
for all $Z, W \in \Gamma(\mathfrak{D}_{\theta})$.
\end{lemma}
\begin{proof}
From Eq. \eqref{phix}, we can write $g(tZ,tW)=g(\phi Z-nZ, tW)$. Hence $g(tZ, tW)=g(Z, \phi tW)$. Using Eqs. \eqref{metric} and \eqref{semi-tsqr}, we obtain Eq. \eqref{semi-gtcos}. Using Eq. \eqref{semi-gtcos} we get Eq. \eqref{semi-gnsin}.
\end{proof}

Next, we will find the necessary and sufficient conditions for involutive and foliation of distributions associated with pointwise semi-slant submanifold of a Lorentzian paracosymplectic manifold.

\begin{lemma}\label{semi-lem2}
If $M$ is a proper pointwise semi-slant submanifold of a Lorentzian paracosymplectic manifold $\bar{M}^{2m+1}$. Then a necessary and sufficient condition for the distribution $\mathfrak{D}_{T}\oplus \{\xi\}$ to be involutive is that the second fundamental form $h$ of $M$ satisfies $h(X, tY)=h(tX,Y)$, for any $X, Y \in \Gamma(\mathfrak{D}_{T}\oplus \{\xi\})$ and $Z \in \Gamma(\mathfrak{D}_{\theta})$.
\end{lemma}
\begin{proof}
In general it is not hard to see that, $g([X,Y], Z)= g(\bar{\nabla}_{X}Y-\bar{\nabla}_{Y}X, Z)$ for any $X, Y, Z \in \Gamma(TM)$.  Above expression by the use of Eq. \eqref{metric} and Lemma \ref{lemsub}, reduced to
\begin{align}\label{semi-lem2_1}
g([X,Y], Z) = g(\phi\bar\nabla_{X}Y, \phi Z)-g(\phi\bar\nabla_Y{X}, \phi Z).
\end{align}
Using Eqs. \eqref{pcmdef} and \eqref{phix} in Eq. \eqref{semi-lem2_1}, we obtain for any $ X, Y \in \Gamma(\mathfrak{D}_{T}\oplus \{\xi\})$ and $Z \in \Gamma(\mathfrak{D}_{\theta})$ that
\begin{align}\label{semi-lem2_2}
g([X,Y], Z) =g(\phi \bar\nabla_{X}Y, tZ)+g(\bar\nabla_{X}tY, nZ)-g(\phi\bar\nabla_Y{X},  tZ)-g(\bar\nabla_Y{tX}, nZ).
\end{align}
Employing Eq. \eqref{symphi}, \eqref{gauss} and \eqref{phix} in Eq. \eqref{semi-lem2_2}, we achieve that
\begin{align}\label{semi-lem2_3}
g([X,Y], Z) =&g( \bar\nabla_{X}Y, t^{2}Z+ntZ)+g(h(X, tY), nZ) \nonumber \\ &-g(\phi\bar\nabla_Y{X},  t^{2}Z+ntZ)-g(h(Y, tX), nZ).
\end{align}
Using the fact that $h$ is symmetric and Eqs. \eqref{gauss}, \eqref{semi-tsqr} in equation \eqref{semi-lem2_2},  we derive that
\begin{align}\label{semi-lem2_3}
\sin^2{\theta} g([X,Y], Z)=g(h(X,tY), nZ)-g(h(tX, Y), nZ). 
\end{align}
Thus, from \eqref{semi-lem2_3}, we conclude that $[X,Y] \in \Gamma(\mathfrak{D}_{T}\oplus \{\xi\})$ if and only if  $h(X, tY)=h( tX, Y)$. Since, $M$ is a proper pointwise semi-slant submanifold and $X, Y, Z$ are non-null vector fields. This completes the proof of the lemma.
\end{proof}

\begin{lemma}\label{semi-lem3}
If $M$ is a proper pointwise semi-slant submanifold of a Lorentzian paracosymplectic manifold $\bar{M}^{2m+1}$. Then a necessary and sufficient condition for the distribution $\mathfrak{D}_{T}\oplus \{\xi\}$ defines a totally geodesic foliation is that metric $g$ in $M$ satisfies $g(A_{ntZ}{Y}, X)=-g(A_{nZ} tY, X)$, for any $X, Y \in \Gamma(\mathfrak{D}_{T}\oplus \{\xi\})$ and $Z \in \Gamma(\mathfrak{D}_{\theta})$.
\end{lemma}
\begin{proof}
For any $X, Y \in \Gamma(\mathfrak{D}_{T}\oplus \{\xi\})$ and $Z \in \Gamma(\mathfrak{D}_{\theta})$, we have from Gauss formula that $g(\nabla_{X}Y, Z) =  g(\bar{\nabla}_{X}Y, Z)$. Employing Eqs. \eqref{metric}, \eqref{pcmdef}, \eqref{shp2form}-\eqref{phin} and Lemma \ref{lemsub} in above expression, we obtain that
\begin{align}\label{semi-lem3_1}
g(\nabla_{X}Y, Z)= g(\bar\nabla_{X}Y, t^{2}Z)+g(h(X,Y), ntZ)+g(h(X, tY), nZ).
\end{align}
Using Eq. \eqref{semi-tsqr} in equation Eq. \eqref{semi-lem3_1}, we arrive at 
\begin{align}\label{semi-lem3_2}
g(\nabla_{X}Y,Z)= \cos^{2}(\theta)g(\nabla_{X}Y, Z)+g(h(X,Y), ntZ)+g(h(X,tY), nZ).
\end{align}
From above equation, we conclude that
\begin{align}\label{semi-lem3_3}
\sin^{2}\theta g(\nabla_{X}Y, Z)=g(h(X,Y), ntZ)+g(h(X, tY), nZ) 
\end{align} 
Thus, from \eqref{semi-lem3_3}, we deduce that $\nabla_{X}Y \in \Gamma(\mathfrak{D}_{T})$ if and only if  $g(h(X,Y), ntZ)+g(h(X, tY), nZ)=0$. Since, $M$ is a proper pointwise semi-slant submanifold and $X, Y, Z$ are non-null vector fields. This completes the proof of the lemma.
\end{proof}

\begin{lemma}\label{semi-lem4}
If $M$ is a proper pointwise semi-slant submanifold of a Lorentzian paracosymplectic manifold $\bar{M}^{2m+1}$ with $\xi \in  \Gamma(TM)$. Then the pointwise slant distribution $\mathfrak{D}_{\theta}$ is involutive if and only if the metric $g$ on $M$ satisfies $g(A_{nW}Z-A_{nZ}W, t X)=g(A_{ntZ}W-A_{ntW}Z, X)$, for any $X \in \Gamma(\mathfrak{D}_{T})$ and $Z,W \in \Gamma(\mathfrak{D}_{\theta})$.
\end{lemma}
\begin{proof}
The proof of this lemma can be achieved by following same steps as used in proving Lemma \ref{semi-lem2}.
\end{proof}

\begin{lemma}\label{semi-lem5}
If $M$ is a proper pointwise semi-slant submanifold of a Lorentzian paracosymplectic manifold $\bar{M}^{2m+1}$ such that $\xi \in  \Gamma(TM)$. Then the pointwise slant distribution $\mathfrak{D}_{\theta}$ defines a totally geodesic foliation if and only if  the metric $g$ on $M$ satisfies $g(A_{nW}tX, Z)=-g(A_{ntW}X, Z)$, for any $X \in \Gamma(\mathfrak{D}_{T})$ and $Z,W \in \Gamma(\mathfrak{D}_{\theta})$.
\end{lemma}
\begin{proof}
The proof of the lemma follow same steps as in Lemma \ref{semi-lem3}.
\end{proof}

\section{Pointwise semi-slant warped product submanifolds}\label{Psswp}
\noindent  In this section, we first define pointwise semi-slant warped product submanifolds $M$, and then examine the existence or non existence results and also derive characterization theorem of such submanifolds in a Lorentzian paracosymplectic manifold $\bar{M}^{2m+1}$ with the structure vector field $\xi$ tangent to $M$.
\begin{definition}
A pointwise semi-slant submanifold $M$ of a Lorentzian almost paracontact manifold $\bar{M}^{2m+1}(\phi,\xi,\eta, g)$ is called a {\it pointwise semi-slant warped product} if it is a warped product of the form: $M_{T}\times_{f} M_{\theta}$ or $M_{\theta}\times_{f}M_{T}$,  where $M_{T}$ (resp., $M_{\theta}$) is invariant (resp., pointwise proper slant) integral submanifolds of $\mathfrak{D}_{T}$ (resp., $\mathfrak{D}_{\theta}$) on $M$ and $f$ is a non-constant positive smooth function on the first factor. 
If the warping function $f$ is constant then a pointwise semi-slant warped product submanifold is said to be a {\it pointwise semi-slant product or trivial product}.
\end{definition}
\noindent From the direct consequence of Theorem \ref{thmwp}, we have the following results for warped product submanifolds when $\xi$ is tangent to second factor;
\begin{proposition}
There doesn't exist a non-trivial pointwise semi-slant warped product submanifold of the form $M=M_{T}\times_{f} M_{\theta}$ of a Lorentzian paracosymplectic manifold $\bar{M}^{2m+1}$ such that the structure vector $\xi$ is tangent to $M_{\theta}$.
\end{proposition}
\begin{proposition}
There doesn't exist a non-trivial pointwise semi-slant warped product submanifold of the form $M=M_{\theta}\times_{f} M_{T}$ of a Lorentzian paracosymplectic manifold $\bar{M}^{2m+1}$ such that the structure vector $\xi$ is tangent to $M_{T}$
\end{proposition}
\noindent Now, we prove an important results for warped product submanifolds when $\xi$ is tangent to first factor;
\begin{theorem}\label{thmwp1}
Let $\bar{M}^{2m+1}$ be a Lorentzian paracosymplectic manifold. Then there does not exist non-trivial pointwise semi-slant warped product submanifold $M=M_{\theta}\times_{f} M_{T}$ of $\bar{M}^{2m+1}$ such that $\xi$ is tangent to $M_{\theta}$.
\end{theorem}
\begin{proof}
We have from Eq. \eqref{gauss}, that $g(\nabla_{X}Z,Y)=g(\bar\nabla_{X}Z, Y)$, for any $X, Y \in \Gamma(M_{T})$ and $Z \in \Gamma(M_{\theta})$. Employing Eqs. \eqref{metric}, \eqref{symphi}, \eqref{pcmdef}, \eqref{phix} and Lemma \ref{lemsub} in right hand side of above expression, we obtain that
\begin{align}\label{thmwp1-1}
g(\nabla_{X}Z,Y)=g(\bar\nabla_{X}t^{2} Z, Y)+g(\bar\nabla_{X}ntZ, Y)+g(\bar\nabla_{X}nZ, \phi Y).               
\end{align}
Using Eqs. \eqref{weingarten}, \eqref{shp2form}, \eqref{semi-tsqr} and the fact $g(Z, Y)=0$ in Eq. \eqref{thmwp1-1}, we obtain that
\begin{align}\label{thmwp1-2}
g(\nabla_{X}Z,Y)= \cos^2\theta g(\bar\nabla_{X}Z,Y)-g(h(X,Y), ntZ)-g(h(X,\phi Y), nZ). 
\end{align}
Applying Eqs.\eqref{gauss} and \eqref{thmwp1-2} in above equation, we conclude that
\begin{align}\label{thmwp1-3}
\sin^2\theta g(\nabla_{X}Z,Y)= -g(h(X,Y), ntZ)-g(h(X,\phi Y), nZ). 
\end{align} 
Interchanging $X$ and $Y$ in Eq. \eqref{thmwp1-3}, we get
\begin{align}\label{thmwp1-4}
\sin^2\theta g(\nabla_{Y}Z, X)= -g(h(X,Y), ntZ)-g(h(Y, \phi X), nZ).
\end{align} 
From Eqs. \eqref{thmwp1-3}, \eqref{thmwp1-4} and Proposition \ref{propmain}, we achieve that 
\begin{align}\label{thmwp1-5}
g(h(X,\phi Y), nZ)=g(h(Y, \phi X), nZ).
\end{align}
On the other hand, by the use of Eqs. \eqref{metric}, \eqref{pcmdef}-\eqref{phix} and Lemma \ref{lemsub}, we arrive at
\begin{align}\label{thmwp1-6}
g(h(X,\phi Y), nZ) = -g(\nabla_{X}Z, Y)+g(\nabla_{X}tZ, \phi Y),
\end{align}
for any $X, Y \in \Gamma(M_{T})$ and $Z \in \Gamma(M_{\theta}).$
Now, from Eq. \eqref{thmwp1-6}, we conclude that the Eq. \eqref{thmwp1-5} hold if and only if  $g(\nabla_{X}tZ, \phi Y)=0$. Moreover, by using Proposition \ref{propmain} and replacing $Z$ by $tZ$, $X$ by $\phi X$ in above expression, we derive that $ t^{2}Z(\ln{f})g(\phi X,\phi Y)=0$. Hence, previous expression in light of Eqs. \eqref{metric}, \eqref{semi-tsqr} and fact that $\eta(X)\eta(Y)=0$ reduced to, $\cos^{2}\theta Z(\ln{f})g(X,Y)=0$. Thus $f$ is constant. Since, $M_{\theta}$ is pointwise proper slant submanifold and $X, Y, Z$ are non-null vector fields.  This completes the proof of the proposition. 
\end{proof}
\begin{remark}
It is no hard to conclude that the Theorem 5.2 in \cite{SY} for $\theta$ globally constant and Theorem 4.1 in \cite{US} for $\theta=\pi/2$  can be treat as the particular cases of the Theorem \ref{thmwp1}.
\end{remark}
\noindent Next, we have an important lemma for later use
\begin{lemma}\label{semiwplem1}
If $M=M_{T}\times_{f} M_{\theta}$ is a non-trivial pointwise semi-slant warped product submanifold  in a Lorentzian paracosymplectic manifold $\bar{M}^{2m+1}$, then
\begin{itemize}
\item[$(a)$] $g(h(X,Z), ntW) = - tX(\ln f)g(tW, Z)+X (\ln f)\cos^{2}\theta g(Z, W)$,
\item[$(b)$] $g(h(tX, Z), nW) = -X(\ln f)g(W, Z)+tX (\ln f)g(Z, tW)$,
\item[$(c)$] $g(h(X,W), ntZ) = - tX(\ln f)g(W, tZ)+X (\ln f)\cos^{2}\theta g(Z, W)$,
\item[$(d)$] $g(h(tX, W), nZ) = -X(\ln f)g(W, Z)+tX (\ln f)g(tZ, W)$,
\end{itemize}
for all $X \in \Gamma(\mathfrak{D}_{T}\oplus \{\xi\})$ and $Z, W \in \Gamma(\mathfrak{D}_{\theta})$.
\end{lemma}
\begin{proof}
From Eqs. \eqref{symphi}, \eqref{pcmdef}, \eqref{phix} and Gauss formulas, we attain that
\begin{align}\label{semiwplem1_1}
g(h(X, W), nZ) = g(\nabla_{X}tW,  Z)+g(\bar\nabla_{X}nW, Z)-g(\nabla_{X}W,  tZ).
\end{align} 
Employing  Proposition \ref{propmain} and Eq. \eqref{symxty} in Eq. \eqref{semiwplem1_1}, we arrive at $$g(h(X, W), nZ)=g(\bar\nabla_{X}nW,  Z).$$ By using Eq. \eqref{weingarten} in right hand side of previous expression, we derive that
\begin{align}\label{semiwplem1_1-1}
g(A_{nZ}W, X) =-g(A_{nW}Z, X).
\end{align} 
Moreover, Eq. \eqref{semiwplem1_1-1}, by replacing $W$ by $tW$ and Eq. \eqref{weingarten} becomes $$g(h(tW, X), nZ) = -g(h(Z, X), ntW).$$ Applying Eqs. \eqref{gauss}, \eqref{weingarten}, \eqref{phix} and the fact that structure is Lorentzian paracosymplectic in above expression, we obtain that
\begin{align}\label{semiwplem1_2}
g(h(Z, X), ntW)= -g(\nabla_{tW}tX,  Z)+g(\nabla_{tW}X,  tZ).
\end{align}
Using Proposition \ref{propmain}, Eq. \eqref{semi-gtcos} and the fact that $\xi$ is orthogonal to $Z, W$ in Eq. \eqref{semiwplem1_2}, we achieve the formula-$(a)$. Thus, replacing $W$ by $tW$ and using Eq. \eqref{semi-tsqr} in \eqref{semiwplem1_2}, we get
\begin{align}\label{semiwplem1_3}
g(h(Z, X), nW) = -tX(\ln f)g(W, Z)+X (\ln f)g(Z, tW).
\end{align}
Now for formula-$(b)$, we first replace $X$ by $\phi X$ in Eq. \eqref{semiwplem1_3}, and then in light of Eqs. \eqref{metric}, \eqref{symphi}, \eqref{pcmdef}, \eqref{shp2form} and fact that $\eta(Z)=0$ we achieve the desired. 
On the other hand,  Using Eq. \eqref{weingarten} and interchanging $Z$ by $tZ$ in Eq. \eqref{semiwplem1_1-1}, we deduce that
\begin{align}\label{semiwplem1_4}
g(h(W, X), ntZ) = -g(h(tZ, X), nW)
\end{align}
Hence, formula-$(c)$ and formula-$(d)$ can be attained with the help of Eq. \eqref{semiwplem1_4} and by following similar steps as used to prove formula-$(a)$ and formula-$(b)$.  This completes the proof of lemma.  
\end{proof}
\noindent Now, we prove an important result as the characterization for pointwise semi-slant warped product submanifold in a Lorentzian paracosymplectic manifold.
 
\begin{theorem}\label{semiwp-thm1}
Let $M\to \bar{M}^{2m+1}$ be an isometric immersion of a submanifold $M$ into a Lorentzian paracosymplectic manifold $\bar{M}^{2m+1}$. Then a necessary and sufficient condition for $M$ to be locally non-trivial pointwise semi-slant warped product submanifold $M_{T}\times_{f} M_{\theta}$ is that the shape operator of $M$ satisfies
\begin{align}\label{semi-Shapecond}
A_{ntW}X+A_{nW}tX=(cos^2{\theta}-1)X (\nu)W,
\end{align}
$\forall X \in \Gamma(\mathfrak{D}_{T}\oplus \{\xi\}), W \in \Gamma(\mathfrak{D}_{\theta})$ and for some function $\nu$ on $M$ such that $Z(\nu)=0$, $Z \in \Gamma(\mathfrak{D}_{\theta})$.
\end{theorem}
\begin{proof}
Let $M$ be a non-trivial pointwise semi-slant warped product submanifold of $\bar{M}^{2m+1}$. Then clearly from formula-$(a)$ and formula-$(b)$ of  lemma \ref{semiwplem1}, we obtain Eq. \eqref{semi-Shapecond}. Since $f$ is a function on $M_{T}$, setting $\mu = \ln{f}$  implies that $Z(\mu)=0$.
Conversely, consider that $M$ is a pointwise semi-slant submanifold of $\bar{M}^{2m+1}$ such that Eq. \eqref{semi-Shapecond} satisfied. By taking inner product of Eq. \eqref{semi-Shapecond} with $X$ and from Lemma \ref{semi-lem3}, we conclude that the integral manifold $M_{T}$ of $\mathfrak{D}_{T}\oplus \{\xi\}$ defines a totally geodesic foliation in $M$. Then by Lemma \ref{semi-lem4}, the distribution $\mathfrak{D}_{\theta}$ is involutive if and only if
\begin{align*}
g(A_{nW}Z-A_{nZ}W, tX)=g(A_{ntZ}W-A_{ntW}Z, X), 
\end{align*}
for all $X \in \mathfrak{D}$ and $Z, W \in \mathfrak{D}_{\theta}$.
Above equation in view of equation  \eqref{shp2form} and fact that $h$, is symmetric can be rearranged as;
\begin{align}\label{semiwp-thm1_0}
g(A_{ntW}X+A_{nW}tX, Z)=g(A_{ntZ}X +g(A_{nZ}tX, W), 
\end{align}
for all $X \in \Gamma(\mathfrak{D}_{T})$ and $Z, W \in \Gamma(\mathfrak{D}_{\theta})$.
Employing formula-$(c)$ and formula-$(d)$ of Lemma \ref{semiwplem1} and Eqs. \eqref{gauss}, \eqref{shp2form} in right hand side of \eqref{semiwp-thm1_0}, we achieve that
\begin{align}\label{semiwp-thm1_1}
g(A_{ntZ}X+A_{nZ}tX, W)=\sin^{2}\theta g(\nabla_W{Z}, X).
\end{align}
Now, taking inner product of Eq. \eqref{semi-Shapecond} with $Z$, we find that
\begin{align}\label{semiwp-thm1_11}
g(A_{ntW}X+A_{nW}tX, Z)=(cos^2{\theta}-1)g(X (\nu)W, Z).
\end{align}
From Eqs. \eqref{semiwp-thm1_0}, \eqref{semiwp-thm1_1} and \eqref{semiwp-thm1_11}, we attain that
\begin{align}\label{semiwp-thm1_111}
g(\nabla_W{Z}, X)=(\cot^2{\theta}-\csc^{2}\theta)X(\nu)g(Z, W),
\end{align}
where, $\nu = \ln{f}$. Hence, from Eq.\eqref{semiwp-thm1_111}, we conclude tha the integrable manifold of $\mathfrak{D}_{\theta}$ is totally umbilical submanifold in $M$ and its mean curvature is non-zero and $Z(\nu)=0$ for all $Z \in \Gamma(\mathfrak{D}_{\theta})$. Thus, from \cite{SH}, we can say that $M$ is a locally non-trivial pointwise semi-slant warped product submanifold of $\bar{M}^{2m+1}$. This completes the proof of the theorem.  
\end{proof}

\end{document}